\documentclass[12pt]{amsart}
\usepackage{amsmath,amsthm,latexsym,amscd,amsbsy,amssymb,amsfonts,fleqn,leqno,
euscript, pb-diagram}

\numberwithin{equation}{section}

\newtheorem{thm}{Theorem}[section]
\newtheorem{pro}[thm]{Proposition}

\newtheorem{cor}[thm]{Corollary}

\begin{document}

%%%%%%%%%%% Begin Topmatter %%%%%%%%%%%%%%%%%

\title[On a  theorem of Arvanitakis]
{On a  theorem of Arvanitakis}

\author{Vesko  Valov}
\address{Department of Computer Science and Mathematics, Nipissing University,
100 College Drive, P.O. Box 5002, North Bay, ON, P1B 8L7, Canada}
\email{veskov@nipissingu.ca}
\thanks{Research supported in part by NSERC Grant 261914-08}


%%%%%%%%%% End topmatter %%%%%%%%%%%%%%%%%%%%%

\begin{abstract}
Arvanitakis \cite{ar} established recently a theorem which is a common generalization
of Michael's convex selection theorem \cite{m} and Dugundji's extension theorem \cite{d}.
In this note we provide a short proof of a more general version of Arvanitakis' result.
\end{abstract}

\maketitle

\markboth{}{Arvanitakis' theorem}

%%%%%%%%%%%%%%%%%%%%%%%%%%%%%%%%%%%%%%%%%%%%%%%%%%%%%%%%%%%%%%%%
%%%%%%%%%%%%%%%%%%%%%%%%%%%%%%%%%%%%%%%%%%%%%%%%%%%%%%%%%%%%%%%%

%%%%%%%%%%%%%%%%%%%%%%%%%%%%%%%%%%%%%%%%%%%%%%%%%%%%%%%%%%%%%%%%%%%%%%

\section{Introduction}

Arvanitakis \cite{ar} established recently the following result
extending both Michael's convex selection theorem \cite{m} and
Dugundji's simultaneous extension theorem \cite{d}:

\begin{thm}\cite{ar}
Let $X$ be a space with property $c$, $Y$ a complete metric space
and $\Phi\colon X\to 2^Y$ a lower semi-continuous set-valued map with
non-empty values. Then for every locally convex complete linear
space $E$ there exists a linear operator $S\colon C(Y,E)\to C(X,E)$
such that
$$S(f)(x)\in\overline{\mathrm{conv}}f(\Phi(x))\hbox{~}\mbox{for all}\hbox{~}x\in X\hbox{~}\mbox{and}\hbox{~}f\in C(Y,E).\leqno{(1)}$$
Furthermore, $S$ is continuous when both $C(Y,E)$ and $C(X,E)$ are
equipped with the uniform topology or the topology of uniform
convergence on compact sets.
\end{thm}

Here, $C(X,E)$ is the set of all continuous maps from $X$ into $E$ (if $E$ is the real line, we write $C(X)$).
We also denote by $C_b(X,E)$ the bounded functions from $C(X,E)$.
Recall that a set-valued map $\Phi\colon X\to 2^Y$ is lower semi-continuous
if the set $\{x\in X:\Phi(x)\cap U\neq\varnothing\}$ is open in $X$ for any open $U\subset Y$.
A space $X$ is said to have property $c$ \cite{ar} if $X$ is paracompact and,
for any space $Y$ and a map $\phi\colon X\to Y$, $\phi$ is
continuous if and only if it is continuous on every compact subspace
of $X$. It is easily seen that the last condition is equivalent to
$X$ being a $k$-space (i.e., the topology of $X$ is determined by
its compact subsets, see \cite{en}).

We provide a short proof of Theorem
1.1. Here is our slightly more general version of Theorem 1.1.

\begin{thm}
Let $X$ be a paracompact space, $Y$ a complete metric space and
$\Phi\colon X\to 2^Y$ a lower semi-continuous set-valued map with
non-empty values. Then: 
\begin{itemize}
\item[(i)] For every locally convex complete linear
space $E$ there exists a linear operator $S_b\colon C_b(Y,E)\to
C_b(X,E)$ satisfying condition $(1)$ such that $S_b$ is continuous
with respect to the uniform topology and the topology of uniform
convergence on compact sets;
\item[(ii)] If $X$ is a $k$-space or $E$ is a Banach space,
$S_b$ can be continuously extended $($with respect to both types of
topologies$)$ to a linear operator $S\colon C(Y,E)\to C(X,E)$
satisfying $(1)$.
\end{itemize}
\end{thm}

Our proof of Theorem 1.2 is based on the idea from a result of Repov\v{s}, P.Semenov and E.Shchepin \cite{rss}
that Michael's zero-dimensional selection theorem yields the convex-valued selection theorem. 

The author would like to express his
gratitude to M. Choban for several valuable suggestions.

%%%%%%%%%%%%%%%%%%%%%%%%%%%%%%%%%%%%%%%%%%%%%%%%%%%%%%%%%%%%%%%%%%%%%%%%%%%%%%%%%%%%%%%%%%%%%%
%%%%%%%%%%%%%%%%%%%%%%%%%%%%%%%%%%%%%%%%%%%%%%%%%%%

\section{Proof of Theorem 1.2}

Let $E$ be a locally convex linear space. We denote by $E^*$ the set
of all continuous linear functionals on $E$ with the topology of
uniform convergence on the weakly bounded subsets of $E$. The second
dual $E^{**}$ is the space of continuous functionals on $E^*$ with
the topology of uniform convergence on the equicontinuous subsets of
$E^*$. It is well known that the canonical map $E\to E^{**}$ is an
embedding, see \cite{Sc}.

We need Banakh's technique \cite{b2} concerning barycenters of some
probability measures. First of all, for every compact space $X$ let
$P(X)$ be the space of all regular probability measures on $X$
endowed with the $w^*$-topology. Each $\mu\in P(X)$ can also be
considered as a continuous linear positive functional on $C(X)$ (the
continuous real-valued functions on $X$ with the uniform convergence
topology) with $\mu(1_X)=1$, where $1_X$ is the constant function on
$X$ having a value one. Recall that for any $\mu\in P(X)$ there
exists a closed nonempty set $\mathrm{supp}(\mu)\subset X$ such that
$\mu(g)=\mu(f)$ for any $f,g\in C(X)$ with
$f|\mathrm{supp}(\mu)=g|\mathrm{supp}(\mu)$, and
$\mathrm{supp}(\mu)$ is the smallest closed subset of $X$ with this
property. If $X$ is a Tychonoff space, we consider the following
subsets of $P(\beta X)$, where $\beta X$ is the \v{C}ech-Stone
compactification of $X$:
$$P_\beta(X)=\{\mu\in P(\beta X):\mathrm{supp}(\mu)\subset X\}$$ and
$$\hat{P}(X)=\{\mu\in P(\beta X): \mu_*(X)=1\}.$$ Here
$\mu_*(X)=\sup\{\mu(B): B\subset X\hbox{~}\mbox{is a Borel subset
of}\hbox{~}\beta X\}$. Every map $h\colon M\to E$ generates a map
$P_\beta(h)\colon P_\beta(M)\to P_\beta(E)$ defined by
$P_\beta(h)(\mu)(\phi)=\mu(\phi\circ h)$, where $\mu\in P_\beta(M)$
and $\phi\in C_b(E)$. In particular, if $i_M\colon M\hookrightarrow
E$ is the inclusion of $M$ into $E$, then  $P_\beta(i_M)$ is
one-to-one and $P_\beta(\delta_x)=\delta_x$ for all $x\in M$
($\delta_x$ is the Dirac measure at the point $x$). The functors
$\hat{P}$ and $P_\beta$ were introduced in \cite{b1} and \cite{ch},
respectively.

Banakh \cite{b2} defined barycenters of measures from $\hat{P}(M)$,
where $M$ is a weakly bounded subset of some locally convex linear
space $E$. For any such $M\subset E$ there exists an affine map
(called a {\em barycenter map}) $b_M\colon\hat{P}(M)\to E^{**}$
which is continuous only when $M$ is bounded in $E$, see
\cite[Theorem 3.2]{b2}. A convex subset $M\subset E$ is called {\em
barycentric} if $b_M(\hat P(M))\subset M$. It was established in
\cite[Proposition 3.10]{b2} that any complete bounded convex subset
of $E$ is barycentric. Since for any $M$ we have
$P_\beta(M)\subset\hat{P}(M)$, we can apply the Banakh arguments
with $\hat{P}(M)$ replaced by $P_\beta(M)$, and this is done in the following
proposition.

\begin{pro}
Let $E$ be a complete locally convex linear space. Then there exists
a not necessarily continuous affine map $b_E\colon P_\beta(E)\to E$
such that $b_E(\mu)\in\overline{\mathrm{conv}}(\mathrm{supp(\mu)})$
for every $\mu\in P_\beta(E)$. Moreover, if $M\subset E$ is a
bounded set then the map $b_E\circ P_\beta(i_M)\colon P_\beta(M)\to
E$ is continuous.
\end{pro}

\begin{proof}
We follow the arguments from \cite{b2}. For every $\mu\in
P_\beta(E)$ we consider the functional $b_E(\mu)\colon E^*\to\mathbb
R$, defined by $b_E(\mu)(l)=\mu(l|\mathrm{supp}(\mu))$, $l\in E^*$.

\textit{Claim. $b_E(\mu)$ is continuous for all $\mu\in
P_\beta(E)$.}

Indeed, suppose $\{l_\alpha\}\subset E^*$ is a net in $E^*$
converging to some $l_0\in E^*$. This means that $\{l_\alpha\}$ is
uniformly convergent to $l_0$ on every weakly bounded subset of $E$.
In particular, $\{l_\alpha\}$ is uniformly convergent to $l_0$ on
$\mathrm{supp}(\mu)$. Consequently, $\{\mu(l_\alpha)\}$ converges to
$\mu(l_0)$.

Therefore, $b_E(\mu)\in E^{**}$ for any $\mu\in P_\beta(E)$. On the
other hand, since $\mathrm{supp}(\mu)\subset E$ is compact and $E$
is complete, $C(\mu)=\overline{\mathrm{conv}}(\mathrm{supp}(\mu))$ is
a compact convex subset of $E$. Then, according to \cite[Proposition
3.10]{b2}, $C(\mu)$ is barycentric and contains $b_E(\mu)$.  So,
$b_E$ maps $P_\beta(E)$ into $E$. The second half of Proposition 2.1
follows from the fact that $E$ is embedded in $E^{**}$ and Theorem
3.2 from \cite{b2}, which (in our situation) states that the map
$b_E\circ P_\beta(i_M)\colon P_\beta(M)\to E^{**}$ is continuous
provided $M$ is bounded in $E$.
\end{proof}

The theory of maps between compact spaces admitting averaging
operators was developed by Pelczy\'{n}ski \cite{p}. For noncompact
spaces we use the following definition \cite{v1}: a surjective
continuous map $f\colon X\to Y$ admits an averaging operator with
compact supports if there exists an embedding $g\colon Y\to
P_\beta(X)$ such that $\mathrm{supp}(g(y))\subset f^{-1}(y)$ for all
$y\in Y$. Then the regular linear operator $u\colon C_b(X)\to
C_b(Y)$, defined by $$u(h)(y)=g(y)(h),\hbox{~}h\in C_b(X),\hbox{~}
y\in Y\leqno{(2)}$$ satisfies $u(\phi\circ f)=\phi$ for any $\phi\in
C_b(Y)$. Such an operator $u$ is called {\em averaging for $f$}.

\begin{pro}
Let $f\colon X\to Y$ be a perfect map admitting an averaging
operator with compact supports and $E$ a complete locally convex
linear space. Then there exists a linear operator $T_b\colon
C_b(X,E)\to C_b(Y,E)$ such that:
\begin{itemize}
\item[(i)] $T_b(h)(y)\in\overline{\mathrm{conv}}\big(h(f^{-1}(y))\big)$ for all
$y\in Y$ and $h\in C_b(X,E)$;
\item[(ii)] $T_b(\phi\circ f)=\phi$ for any $\phi\in C_b(Y,E)$;
\item[(iii)] $T_b$ is continuous when both $C_b(X,E)$ and $C_b(Y,E)$
are equipped with the uniform topology or the topology of uniform
convergence on compact sets.
\end{itemize}
Moreover, if $Y$ is a $k$-space or $E$ is a Banach space, $T_b$ can be extended to a linear
operator $T\colon C(X,E)\to C(Y,E)$ satisfying conditions $(i) -
(iii)$ with $C_b(X,E)$ and $C_b(Y,E)$ replaced, respectively, by
$C(X,E)$ and $C(Y,E)$.
\end{pro}

\begin{proof}
A similar statement to the first part was proved in
\cite[Proposition 3.1]{v1}.
We fix an embedding $g\colon Y\to P_\beta (X)$ with
$\mathrm{supp}(g(y))\subset f^{-1}(y)$, $y\in Y$. For every $h\in
C_b(X,E)$ consider the map $$T_b(h)\colon Y\to E, \hbox{~}
T_b(h)(y)=b_E(P_\beta(i_{h(X)})(\nu_y)),\leqno{(3)}$$ where
$i_{h(X)}\colon h(X)\hookrightarrow E$ is the inclusion and
%and $P_\beta(i_{h(X)})\colon P_\beta(h(X))\to P_\beta(E)$.%
$\nu_y\in P_\beta(h(X))$ is the measure $P_\beta(h)(g(y))$.
According to Proposition 2.1, $T_b(h)$ is continuous (recall that
$h(X)\subset E$ is bounded). It also follows from the definition of
the map $b_E$ that $T_b$ is linear. Since
$\mathrm{supp}(g(y))\subset f^{-1}(y)$ and
$\mathrm{supp}\big(P_\beta(i_{h(X)})(\nu_y)\big)\subset
h(f^{-1}(y))$, $y\in Y$, we have
$b_E(P_\beta(i_{h(X)})(\nu_y))\subset\overline{\mathrm{conv}}(h(f^{-1}(y)))$
(see Proposition 2.1). So, $T_b$ satisfies condition $(i)$.
Moreover, $T_b(h)$ belongs to $C_b(Y,E)$ because
$T_b(h)(y)\subset\overline{\mathrm{conv}}(h(X))$ for all $y\in Y$. It
follows directly from $(2)$ and $(3)$ that $T_b$ satisfies condition
$(ii)$. To prove $(iii)$, assume $K\subset Y$ is compact and let
$W_1=\{\phi\in C_b(Y,E):\phi(K)\subset V_1\}$, where $V_1$ is a
convex neighborhood of $0$ in $E$. Obviously, $W_1$ is a
neighborhood of the zero function in $C_b(Y,E)$. Take a convex
neighborhood $V_2$ of $0$ in $E$ with $\overline{V}_2\subset V_1$
and let $W_2=\{h\in C_b(X,E): h(H)\subset V_2\}$, $H=f^{-1}(K)$.
Since $H$ is compact (recall that $f$ is a perfect map), $W_2$ is a
neighborhood of $0$ in $C_b(X,E)$. Moreover, for all $y\in Y$ and
$h\in W_2$ we have
$T_b(h)(y)\subset\overline{\mathrm{conv}}(h(H))\subset\overline{V}_2\subset
V_1$. So, $T_b(W_2)\subset W_1)$. This provides continuity of $T_b$
with respect to the topology of uniform convergence on compact sets.
Similarly, one can show that $T_b$ is also continuous with respect to the uniform
topology.

Assume that $Y$ is a $k$-space and $h\in C(X,E)$. Then formula
$(3)$ provides a map $T(h)\colon Y\to E$ satisfying conditions $(i)$
and $(ii)$. We need to show that $T(h)$ is continuous on every
compact set $L\subset Y$. And this follows from Proposition 2.1
because the set $h(f^{-1}(L))\subset E$ is compact. So, $T(h)$ is
continuous and, obviously, $T(h)=T_b(h)$ for all $h\in C_b(X,E)$.
Continuity of $T$ follows from the same arguments we used to prove
continuity of $T_b$.

If $E$ is a Banach space, then every $T(h)$, $h\in C(Y,E)$, is continuous
without the requirement $Y$ to be a $k$-space. Indeed, we fix $y_0\in Y$ and
$h\in C(X,E)$. Let $V$ be a bounded closed neighborhood of $h(f^{-1}(y_0))$ in $E$.
Then $h^{-1}(V)$ is a neighborhood of $f^{-1}(y_0)$ and, since $f$ is a perfect map,
there exists a closed neighborhood $U$ of $y_0$ in $Y$ with $W=f^{-1}(U)\subset h^{-1}(V)$.
Then, according to Proposition 2.1, the map  $b_E\circ P_\beta(i_V)\colon P_\beta(V)\to
E$ is continuous. On the other hand $P_\beta(h)$ maps continuously $P_\beta(W)$ into $P_\beta(V)$ and
$g(U)\subset P_\beta(W)$ is homeomorphic to $U$ (recall that $g$ is an embedding of $Y$ into $P_\beta(X))$.
Hence, $T(h)$ is continuous on $U$. Because $U$ is a neighborhood of $y_0$ in $Y$, this implies
continuity of $T(h)$ at $y_0$.
\end{proof}

\textit{Proof of Theorem $1.2$.}  Suppose $X$, $Y$, $\Phi$ and $E$
satisfy the hypotheses of Theorem 1.2. By \cite{rss} (see also \cite{rs}), there exists a
zero-dimensional paracompact space $X_0$ and a perfect surjection
$f\colon X_0\to X$ admitting a regular averaging operator. By
Proposition 2.2, there exists a linear operator $T_b\colon
C_b(X_0,E)\to C_b(X,E)$ satisfying conditions $(i) - (iii)$. The map
$\tilde{\Phi}\colon X_0\to 2^Y$,
$\tilde{\Phi}(z)=\overline{\Phi(f(z))}$, is lower semi-continuous
with closed non-empty values in $Y$. So, according to the Michael's
0-dimensional selection theorem \cite{m1}, $\tilde{\Phi}$ has a
continuous selection $\theta\colon X_0\to Y$. Now, we define the
linear operator $S_b\colon C_b(Y,E)\to C_b(X,E)$ by
$S_b(h)=T_b(h\circ\theta)$, $h\in C_b(Y,E)$. Obviously,
$\theta(f^{-1}(x))\subset\overline{\Phi(x)}$ for every $x\in X$.
Then, according to $(i)$, for all $h\in C_b(Y,E)$ and $x\in X$ we
obtain
$$S_b(h)(x)=T_b(h\circ\theta)(x)\subset
\overline{\mathrm{conv}}\big((h\circ\theta)(f^{-1}(x))\big)\subset\overline{\mathrm{conv}}\big(h(\Phi(x))\big).$$
Continuity of $S_b$ follows from continuity of $T_b$ and the map
$\theta$.

If $X$ is a $k$-space or $E$ is a Banach space, the operator $T_b$ can be extended to a
linear operator $T\colon C(X_0,E)\to C(X,E)$ satisfying conditions
$(i) - (iii)$ from Proposition 2.2. Then $S\colon C(Y,E)\to C(X,E)$,
$S(h)=T(h\circ\theta)$, is the required linear operator extending
$S_b$.

\section{Remarks}

Let us show first that Theorem 1.2 implies Michael's selection theorem.
Assume $X$ is paracompact, $Y$ is a Banach space and $\Phi\colon X\to 2^Y$ a lower semi-continuous
map with closed convex values. Then, by Theorem 1.2
there exists a linear operator $S\colon C(Y,Y)\to
C(X,Y)$ satisfying condition $(1)$. Since the values of $\Phi$ are convex and closed,
condition $(1)$ yields that $S(id_Y)(x)\in\Phi(x)$ for all $x\in X$, where $id_Y$ is the identity on $Y$.
Hence, $S(id_Y)$ is a continuous selection for $\Phi$.

The original Dugundji theorem \cite{d} states that if $X$ is a metric space, $A\subset X$ its closed
subset and $E$ a locally convex linear space, then there exists a linear operator $S:C(A,E)\to C(X,E)$
such that $S(f)$ extends $f$ for any $f\in C(A,E)$. When both $E$ and $A$ are complete, Dugundji theorem can
be derived from Theorem 1.2. Indeed, let $A$ be a completely metrizable closed subset of a paracompact $k$-space $X$
and $E$ a complete locally convex linear space. Consider the set-valued map $\Phi\colon X\to 2^A$, $\Phi(x)=\{x\}$ if $x\in A$
and $\Phi(x)=A$ if $x\not\in A$. Let  $S\colon C(A,E)\to C(X,E)$ be a linear operator satisfying $(1)$. Then
$S(f)(x)=f(x)$ for all $f\in C(A,E)$ and $x\in A$. So, $S$ is an extension operator. If $X$ is not necessarily a $k$-space,
there exists an extension linear operator $S_b\colon C_b(A,E)\to C_b(X,E)$.

Heath and Lutzer \cite[Example 3.3]{hl} provided an example of a paracompact $X$ and a closed set $A\subset X$  homeomorphic to the
rational numbers such that there is no extension operator from $C(A)$ to $C(X)$. This space is the Michael's line, i.e.,
the real line with topology consisting of all sets of the form $U\cup V$, where $U$ is an open
subset of the rational numbers and $V$ is a subset of the irrational numbers. It is easily seen that this a $k$-space.
So, the assumption in the above result $A$ to be completely metrizable is essential.

The original Dugundji theorem with $E$ complete can be derived from Proposition 2.2 and the well known fact that every closed subset of a zero-dimensional
metric space $X$ is a retract of $X$, see for example \cite[Problem 4.1.G]{e}. Indeed, assume $X$ is a metric space and $A\subset X$ its closed
subset. By \cite{ch}, there exists a zero-dimensional metric space $X_0$ and a perfect surjection $f\colon X_0\to X$ admitting an
averaging operator. Let $A_0=f^{-1}(A)$
and $r\colon X_0\to A_0$ be a retraction. Define the linear operator $S\colon C(A,E)\to C(X,E)$ by $S(h)=T(h\circ f\circ r)$, where
$E$ is a complete locally convex linear space, $h\in C(A,E)$ and $T\colon C(X_0,E)\to C(X,E)$ is the operator from Proposition 2.2. It
follows from Proposition 2.2(i) that $S$ is an extension operator.

The proof of Theorem 1.2 is based on two main facts: the 0-dimensional Michael's selection theorem and the Repov\v{s}-Semenov-Shchepin result \cite{rss} that each paracompactum is a continuous image of under a perfect map admitting an averaging operator. So, the 0-dimensional Michael's selection theorem implies not only the convex-valued section theorem, but it also implies the Dugundji extension theorem. Actually we have the
following corollary from Proposition 2.2 ($\mathrm{Sel(\Phi)}$ denotes all continuous selections for $\Phi$).  

\begin{cor}
Let $f\colon X\to Y$ be a perfect map admitting an averaging
operator with compact supports and $E$ a Banach space. Suppose $\Phi\colon Y\to 2^E$
is a lower semi-continuous set-valued map  with closed convex non-empty values. Then
there exists an affine map from $\mathrm{Sel(\Phi\circ f)}$ to $\mathrm{Sel(\Phi)}$ which is continuous
when both $\mathrm{Sel(\Phi\circ f)}$ and $\mathrm{Sel(\Phi)}$
are equipped with the uniform topology or the topology of uniform
convergence on compact sets.
\end{cor}

%%%%%%%% Bibliography %%%%%%%%%%%%%%%%%%%%%%%%%%


\begin{thebibliography}{999}

\bibitem{aa}
S.~Argiros and A.~Arvanitakis, \textit{A characterization of regular
averaging operators and its consequences}, Studia Math. \textbf{151,
3} (2002), 207--226.

\bibitem{ar}
A.~Arvanitakis, \textit{A simultaneous selection theorem}, preprint.

\bibitem{b1}
T.~Banakh, \textit{Topology of spaces of probability measures $I$:
The functors $P_\tau$ and $\hat{P}$}, Mat. Studii \textbf{5} (1995),
65-87 (in Russian).

\bibitem{b2}
T.~Banakh, \textit{Topology of spaces of probability measures $II$:
Barycenters of probability Radon measures and metrization of the
functors $P_\tau$ and $\hat{P}$}, Mat. Studii \textbf{5} (1995),
88-106 (in Russian).

%\bibitem{bv2}
%T.~Banakh and V.~Valov, \textit{General Position Properties in
%Fiberwise Geometric Topology}, arXiv:1001.2494v1[math.GT].

\bibitem{ch}
A.~Chigogidze, \textit{Extension of normal functors}, Vestnik Mosk.
Univ. Ser. I Mat. Mekh. \textbf{6} (1984), 23--26 (in Russian).

\bibitem{ch}
M.~Choban, \textit{Topological structures of subsets of topological
groups and their quotient spaces}, Mat. Issl. \textbf{44} (1977),
117--163 (in Russian).

\bibitem{d}
J.~Dugundji, \textit{An extension of Tietze's theorem}, Pacific J.
Math. \textbf{1} (1951), 353--367.

\bibitem{en}
R.~Engelking, \textit{General topology}, Polish Scientific
Publishers, Warszawa (1977).

\bibitem{e}
R.~Engelking, \textit{Theory of Dimensions: Finite and Infinitite},
Heldermann, Lemgo, 1995.


\bibitem{hl}
R.~Heath and D.~Lutzer, \textit{Dugundji extension theorems for linearly ordered spaces},
Pacif. J. Math. \textbf{55, 2} (1974), 419--425.


\bibitem{m}
E.~Michael, \textit{Continuous selections I}, Ann. of Math.
\textbf{63} (1956), 361--382.

\bibitem{m1}
E.~Michael, \textit{Continuous selections II}, Ann. of Math.
\textbf{64} (1956), 362--380.

\bibitem{p}
A.~Pelczy\'{n}ski, \textit{Linear extensions, linear averagings, and
their applications to linear topological classification of spaces of
continuous functions}, Dissert. Math. \textbf{58} (1968), 1--89.

\bibitem{rs}
D.~Repov\v{s} and P.~Semenov, \textit{Continuous selections of
multivalued mappings}, Mathematics and its Applications, 455. Kluwer
Academic Publishers, Dordrecht, 1998.

\bibitem{rss}
D.~Repov\v{s}, P.~Semenov and E.~Shchepin, \textit{On
zero-dimensional Millutin maps and Michael selection theotems},
Topology Appl. \textbf{54} (1993), 77--83.

\bibitem{Sc}
H.~Schaefer, \textit{Topological vector spaces}, Graduate Texts in
Mathematics,Vol. 3. Springer-Verlag, New York-Berlin, 1971.

%\bibitem{m1}
%E.~Michael, \textit{Selected selection theorems}, Americam Math.
%Montly \textbf{63} (1956), 233--238.

%\bibitem{m2}
%E.~Michael, \textit{A note on paracompact spaces}, Proc. Amer. Math.
%Soc. \textbf{4} (1953), 831--638.

%\bibitem{m3}
%E.~Michael, \textit{A theorem on semi-continuous set-valued
%functions}, Duke Math. J. \textbf{26, 4} (1959), 647--656.
\bibitem{v1}
V.~Valov, \textit{Linear operators with compact supports,
probability measures and Milyutin maps}, J. Math. Anal. Appl.
\textbf{370} (2010), 132--145.

\end{thebibliography}
\end{document}